\documentclass[10pt]{article}

\usepackage{cmap}
\usepackage[utf8]{inputenc}
\usepackage[T1]{fontenc}
\usepackage{lmodern}
\usepackage{fullpage}
\usepackage{amsmath,amssymb,amsfonts,amsthm,amsbsy,mathtools}
\usepackage{thmtools}
\usepackage{graphicx,xcolor}
\usepackage{enumerate}
\usepackage{relsize}
\usepackage{hyperref}
\hypersetup{
    colorlinks=true,
    linkcolor=blue,
    citecolor=darkgray
}

\usepackage{easy-todo}

\usepackage{fancyvrb}
\fvset{formatcom=\color{black}, fontsize=\relsize{-0.5}}

\newtheorem{theorem}{Theorem}
\newtheorem{corollary}[theorem]{Corollary}
\newtheorem{proposition}[theorem]{Proposition}
\newtheorem{lemma}[theorem]{Lemma}

\numberwithin{theorem}{section}
\numberwithin{corollary}{section}
\numberwithin{proposition}{section}
\numberwithin{lemma}{section}

\theoremstyle{definition}
\newtheorem{definition}[theorem]{Definition}
\newtheorem{conjecture}[theorem]{Conjecture}

\theoremstyle{remark}
\newtheorem{remark}[theorem]{Remark}

\newtheorem{examplex}[theorem]{Example}
\newenvironment{example}
  {\pushQED{\qed}\examplex}
  {\popQED\endexamplex}

\newcommand{\mB}{\mathcal{B}}
\newcommand{\mS}{\mathcal{S}}
\newcommand{\mL}{\mathcal{L}}
\newcommand{\R}{\mathbb{R}}
\newcommand{\C}{\mathbb{C}}
\newcommand{\Q}{\mathbb{Q}}
\newcommand{\Z}{\mathbb{Z}}
\newcommand{\N}{\mathbb{N}}

\author{Stephen Melczer \and Marc Mezzarobba}
\date{\today}

\title{Sequence Positivity Through Numeric Analytic Continuation: Uniqueness of the Canham Model for Biomembranes}
\begin{document}
\maketitle

\begin{abstract}
We prove solution uniqueness for the genus one Canham variational problem arising in the shape prediction of biomembranes. The proof builds on a result of Yu and Chen that reduces the variational problem to proving non-negativity of a sequence defined by a linear recurrence relation with polynomial coefficients. We combine rigorous numeric analytic continuation of D-finite functions with classic bounds from singularity analysis to derive an effective index where the asymptotic behaviour of the sequence, which is positive, dominates the sequence behaviour. Positivity of the finite number of remaining terms is then checked computationally.
\end{abstract}

\section{Introduction}
\label{introduction}

An influential biological model of Canham~\cite{Canham1970} predicts the preferred shapes of biomembranes, such as blood cells, by solving a variational problem involving mean curvature. For a fixed genus%
\footnote{The model fixes a genus as experimental observations have found no topological changes in surfaces whose external systems evolve, at accessible time-scales.
Although genus zero biomembranes are more commonly observed in living organisms, genus one membranes can be observed under the microscope in laboratory settings~\cite[Sect.~4]{MichaletBensimon1995}.}
$g$ and constants~$a_0$ and $v_0$ determined by physical details, such as ambient temperature, the model of Canham asks one to find, among all orientable closed surfaces of genus $g$ of prescribed area $a_0$ and volume $v_0$, a surface $S$ minimizing the \emph{Willmore energy}
\begin{equation} W(S) = \int_S H^2 \, dA, \end{equation}
where $H$ is the mean curvature.
Other popular models of membrane shape prediction due to Helfrich~\cite{Helfrich1973} and Evans~\cite{Evans1974} ask for minimization of~$W(S)$ under different constraints (the Helfrich model adds a constraint, while the Evans model removes the volume constraint).
Because $W(S)$ is scaling invariant, prescribing the area $A(S)$ and volume $V(S)$ of the surface turns out to be equivalent to prescribing the \emph{isoperimetric ratio}
\[ \iota(S) = \pi^{1/6}\frac{\sqrt[3]{6V(S)}}{\sqrt{A(S)}} = \iota_0. \]
The \emph{isoperimetric inequality} states that $\iota(S) \in (0,1]$, with $\iota(S)=1$ achieved uniquely for the sphere.

The existence of a solution to the Canham model in genus $g=0$ and any $\iota_0 \in (0,1]$ was shown by Schygulla~\cite{Schygulla2012}, while Keller et al.~\cite{KellerMondinoRiviere2014} proved existence of solutions for higher genus and some values of $\iota_0$ between zero and one. Due to the apparent uniqueness of biomembrane shapes observed in experimental settings, it is natural to ask whether such a prediction model admits a unique solution.
Computational investigations of solution existence and uniqueness for the Canham model have been carried out in Seifert~\cite{Seifert1997} and Chen et al.~\cite{ChenYuBroganYangZigerelli2018}.
Recent work of Yu and Chen~\cite{YuChen2020} further investigates the uniqueness problem, showing that there are Canham models with non-homothetic solutions in genus $g\geq2$ and conjecturing solution uniqueness up to homothetic transformation in all genus zero and genus one settings.

\begin{conjecture}[{Yu and Chen~\cite[Conjecture~1.1]{YuChen2020}}]
\label{conj:YuChen1}
Let $\tau = \frac{3}{2^{5/4}\sqrt{\pi}}$. Up to homothetic transformation,
\begin{enumerate}[(i)]
	\item If $g=0$ and $\iota_0\in(0,1]$, or if $g=1$ and $\iota_0\in(0,\tau]$, then the Canham model has a unique solution given by a surface of revolution;
	\item If $g=1$ and $\iota_0 \in [\tau,1)$ then the Canham model has a unique solution defined by the stereographic image in $\R^3$ of the \emph{Clifford torus}
	\[ \left\{ \frac{1}{\sqrt{2}}{\Big[}\cos u,\sin u, \cos v, \sin v{\Big]}^T : u,v \in [0,2\pi] \right\} \subset \mathbb{S}^3. \]
\end{enumerate}
\end{conjecture}

The work of Keller et al.~\cite{KellerMondinoRiviere2014} mentioned above proves the existence of a solution to the genus one Canham model only for values of $\iota_0 \in [\tau,1)$, corresponding to Conjecture~\ref{conj:YuChen1}(ii).
After giving heuristic arguments for why Conjecture~\ref{conj:YuChen1} should hold, Yu and Chen reduce proving Conjecture~\ref{conj:YuChen1}(ii) to showing that a certain sequence of rational numbers has positive terms. More specifically, let~$(d_n)$ be the unique sequence with initial terms 
\[ (d_0,\dots,d_6) = \left(72, 1932, 31248, \frac{790101}{2}, \frac{17208645}{4}, \frac{338898609}{8}, \frac{1551478257}{4}\right) \] 
satisfying the explicit order seven linear recurrence relation 
\begin{equation}
  \label{eq:introrec}
  \sum_{i=0}^7 r_i(n)d_{n+i} = 0, \qquad r_j(n) \in \Z[n]
\end{equation}
defined in~\eqref{eq:mainRec} of the appendix.

\begin{conjecture}[Yu and Chen~\cite{YuChen2020}]
\label{conj:YuChen2}
All terms of the sequence $(d_n)$ defined by~\eqref{eq:mainRec} are positive.
\end{conjecture}

\begin{proposition}[{Yu and Chen~\cite[Prop.~1.2]{YuChen2020}}]
If all terms of the sequence $(d_n)$ defined by~\eqref{eq:mainRec} are positive then Conjecture~\ref{conj:YuChen1}(ii) holds.
\end{proposition}

The main result of this paper is to prove Conjecture~\ref{conj:YuChen2}, thus completing the uniqueness proof of the Clifford torus in the Canham model.

\begin{theorem}
\label{thm:mainthm}
All terms of the sequence defined by~\eqref{eq:mainRec} are positive.
\end{theorem}

\begin{corollary}
\label{cor:mainCor}
Up to homothetic transformation, any Canham model with genus $g=1$ and fixed isoperimetric ratio $\iota_0 \in [\tau,1)$ has a unique solution defined by the stereographic image in $\R^3$ of the Clifford torus.
\end{corollary}

Our proof aims to illustrate a general method to obtain asymptotic approximations with error bounds of sequences defined by recurrence relations of the type~\eqref{eq:introrec},
based on analytic combinatorics and rigorous numerics.
The method is implicit in the work of Flajolet and collaborators~\cite{FlajoletPuech1986,FlajoletOdlyzko1990,FlajoletSedgewick2009},
however, to the best of our knowledge, it has never been detailed or used in published work.
We also aim to illustrate the computational tools available to compute these bounds on practical applications.
A Sage notebook containing our calculations can be found at
\url{http://doi.org/10.5281/zenodo.4274505}
or, for an interactive version,
\begin{center}
\url{https://mybinder.org/v2/zenodo/10.5281/zenodo.4274505/?filepath=Positivity.ipynb}.
\end{center}

A more direct proof of Corollary~\ref{cor:mainCor} is also possible.
Indeed, the argument of Yu and Chen shows that it follows from the weaker condition that the power series $\sum_{n \geq 0} d_nz^n$ is positive for all $z \in (0,3-2\sqrt{2})$.
As outlined in Remark~\ref{rem:fpositive} of Section~\ref{sec:finalremarks}, this fact can be established using variants some of the arguments involved in the proof of Theorem~\ref{thm:mainthm}, without going through a full proof of positivity of the coefficient sequence.

\subsection{A Short History of Sequence Positivity}
\label{sec:history}

The study of positivity for recursively defined sequences has a long history, in combinatorics as well as mathematics and computer science more broadly. A full accounting of works on this topic would be more than enough to fill a survey paper (or textbook) so we aim only to highlight some specific problems close to our results and approach.

One of the oldest outstanding problems in this area is the so-called Skolem problem for \emph{C-finite sequences} (sequences satisfying linear recurrence relations with constant coefficients). Skolem's problem asks one to decide, given a C-finite sequence encoded by a linear recurrence with constant coefficients and a sufficient number of initial terms, whether any term in the sequence is zero. Because the term-wise product $(a_nb_n)$ of any two C-finite sequences $(a_n)$ and $(b_n)$ is also C-finite, Skolem's problem for a real sequence $(a_n)$ can be reduced to deciding when the C-finite sequence $(a_n^2)$ has only positive terms. Although the general term of a C-finite sequence can be algorithmically represented as an explicit finite sum involving powers of algebraic numbers, decidability of positivity has essentially been open since Skolem's work~\cite{Skolem1934} characterizing zero index sets of C-finite sequences in the 1930s. Skolem's problem has received great attention in the theoretical computer science literature, as the counting sequences of regular languages are always C-finite. See Kenison et al.~\cite{KenisonLiptonOuaknineWorrell2020} for an overview of the topic, together with some recent progress. 

For more general recurrence relations, Gerhold and Kauers~\cite{GerholdKauers2005} introduced a computer algebra procedure that tries to find an inductive proof of positivity in which the induction step can be automatically established using algorithms for cylindrical algebraic decomposition.
The special case of linear recurrence relations with \emph{polynomial} coefficients---like Yu and Chen's---was further studied by Kauers and Pillwein~\cite{KauersPillwein2010,Pillwein2013}, who gave extensions of the basic technique and sufficient conditions for termination%
\footnote{Thomas Yu informed us that the method described by Kauers and Pillwein fails in practice to prove positivity of our sequence~$d_n$, though it does apply to simpler sequences used in intermediary computations by Yu and Chen~\cite{YuChen2020}.}.
Another computer algebra method, due to Cha~\cite{Cha2014}, sometimes allows one to express solution sequences as sums of squares.
The present paper indirectly builds on a different family of algorithms, going back to Cauchy~\cite{Cauchy1842}, that provide \emph{upper} bounds on the magnitude of coefficients of power series solutions to various kinds of functional equations.
Singularity analysis, in a sense, allows us to ``turn upper bounds into two-sided ones'' and use them to derive positivity results.
We refer the interested reader to~\cite[Sec.~2.1]{Mezzarobba2019} for further references.

Finally, we mention that positivity of power series coefficients has long been of interest to analysts (in contexts not so different from the variation problem at the heart of Canham's model). For instance, during their 1920s work on solution convergence for finite difference approximations to the wave equation, Friedrichs and Lewy attempted to prove positivity of a three-dimensional sequence defined as the power series coefficients of a trivariate rational function; positivity was shown by Szeg{\"o}~\cite{Szego1933} using properties of Bessel functions. Askey and Gasper~\cite{AskeyGasper1972} and Askey~\cite{Askey1974} detail this problem and additional ones in a similar vein, and a vast generalization of Szeg{\"o}'s result was given by Scott and Sokal~\cite{ScottSokal2014}.

\section{Singular Behaviour and Eventual Positivity}
In order to reason about the sequence $(d_n)$ we encode it by its \emph{generating function},
\[ f(z) = \sum_{n\geq0}d_nz^n.\]
Because $(d_n)$ satisfies a linear recurrence relation with polynomial coefficients, $f(z)$ satisfies a linear differential equation with polynomial coefficients, and such a differential equation can be determined automatically: see~\cite[Sect.~VII. 9]{FlajoletSedgewick2009} or~\cite[Ch.~14]{BostanChyzakGiustiLebretonLecerfSalvySchost2017} for details. In this case, $f(z)$ satisfies a third-order differential equation
\[ c_3(z)F'''(z) + c_2F''(z) + c_1F'(z) + c_0F(z) = 0, \qquad c_j(z) \in \Z[z] \]
given explicitly in~\eqref{eq:mainDiff} of the appendix.

Because~\eqref{eq:mainDiff} is a \emph{linear} differential equation its formal power series solutions form a complex vector space of dimension at most three. Our particular generating function solution $F(z)=f(z)$ can be uniquely specified among the formal power series solutions of~\eqref{eq:mainDiff} by a finite number of initial conditions $F(0)=d_0, F'(0)=d_1, \dots$. Although $f(z)$ cannot be expressed easily in closed form, we can leverage its representation as a solution of~\eqref{eq:mainDiff} to compute enough information to prove positivity of $(d_n)$. Our computations are carried out in the Sage%
\footnote{
Available at
\url{http://sagemath.org/}.
We use SageMath version~9.1
(\href{https://doi.org/10.5281/zenodo.4066866}{\texttt{doi:10.5281/zenodo.4066866}},
Software Heritage persistent identifier
\href{https://archive.softwareheritage.org/swh:1:rel:5e11f7bf8344447a93ae043b915f3b25e62b7ed6/}%
{\texttt{swh:1:rel:5e11f7bf8344447a93ae043b915f3b25e62b7ed6}}%
).%
}
ore\_algebra%
\footnote{
Available at \url{https://github.com/mkauers/ore_algebra/}.
We use git revision \texttt{2d71b5} (Software Heritage persistent identifier
\href{https://archive.softwareheritage.org/swh:1:rev:2d71b50ebad81e62432482facfe3f78cc4961c4f/}%
{\texttt{swh:1:rev:2d71b50ebad81e62432482facfe3f78cc4961c4f}}).%
}
 package~\cite{KauersJaroschekJohansson2015,Mezzarobba2016}.

\begin{example}
The ore\_algebra package represents linear differential equations such as~\eqref{eq:mainDiff} as \emph{Ore polynomials}: essentially, polynomials in two non-commuting variables which encode linear differential operators. For instance, to load the package and encode the equation~\eqref{eq:mainDiff} one can enter

\begin{Verbatim}
sage: from ore_algebra import *
sage: Pols.<z> = PolynomialRing(QQ); Diff.<Dz> = OreAlgebra(Pols)
sage: deq = (25165779*z^15 - ... - 25165779*z^2)*Dz^3 + ... + (6341776308*z^12 - ... + 2701126946)
\end{Verbatim}
where each $\dots$ represents explicit input which is truncated here for readability. A term of the form \Verb#Dz^k# represents an operator taking $f(z)$ to its $k$th derivative.
\end{example}

We prove positivity of $d_n$ through comparison with its asymptotic behaviour.
We will soon see that the power series~$f(z)$ is convergent, and hence defines an analytic function, in a neighbourhood of zero in the complex plane; we also denote this analytic function by $f(z)$.
Dominant asymptotics are calculated using the transfer method of Flajolet and Odlyzko~\cite{FlajoletOdlyzko1990}, which shows how asymptotic behaviour of $d_n$ is linked to the singular behaviour of the analytic function~$f(z)$. In particular, to determine asymptotic behaviour of $d_n$ it is enough to identify the singularity of $f(z)$ with minimal modulus (in this case there is only one), compute a singular expansion of $f(z)$ in a region near this singularity, then transfer information from the dominant terms of this singular expansion directly into dominant asymptotic behaviour of $d_n$. 

The singular behaviour of $f(z)$ is constrained by the fact that it satisfies~\eqref{eq:mainDiff}.
The classical Cauchy existence theorem for analytic differential equations implies that analytic solutions of~\eqref{eq:mainDiff} can be analytically continued to any simply connected domain $\Omega \subseteq \C$ where the leading coefficient
\[ c_3(z) = 8388593z^2(z+1)^2(z-1)^3(z^2-6z+1)^2(3z^4 - 164z^3 + 370z^2 - 164z + 3) \]
of~\eqref{eq:mainDiff} does not vanish.
In fact, only a subset of these zeroes will be singularities of the solutions to~\eqref{eq:mainDiff}.

\begin{lemma} \label{lem:singularities}
If $\zeta \in \C$ is a singularity of a solution to~\eqref{eq:mainDiff} then $\zeta$ lies in the set
\[ \Xi = \{0,1,3\pm2\sqrt{2}\}. \]
\end{lemma}

\begin{proof}
Following the Sage code above, the command
\begin{Verbatim}
sage: desing_deq.desingularize()
\end{Verbatim}
returns an order 7 linear differential equation, satisfied by all solutions of~\eqref{eq:mainDiff}, whose leading coefficient polynomial is $C(z) = (z - 1)^2z^2(z^2 - 6z + 1)^2$.
The stated conclusion then follows from the Cauchy existence theorem applied to this differential equation, as the roots of $C$ form the set $\Xi$.
\end{proof}

For a given $\zeta \in \Xi$ some solutions of the differential equation~\eqref{eq:mainDiff} may admit convergent power series expansions, while others may admit $\zeta$ as a singularity.
In the present case, for each $\zeta \in \Xi$, the Fuchs criterion~\cite[§55]{Poole1936} shows that $\zeta$ is a \emph{regular singular point} of the equation, meaning the equation admits a full basis of formal solutions of the form
\begin{equation}
  \label{eq:regsing}
  g(z) = z^{\nu} \sum_{n=0}^{\infty} \left(\sum_{k=0}^{\kappa} C_{n,k} \log^k\frac{1}{1-z/\zeta}\right)\left(z-\zeta\right)^{n},
\end{equation}
where $\nu \in \overline{\Q}$ (the field of algebraic numbers), $\kappa\in\N$, and each $C_{n,k}\in\C$.
In addition, the power series $\sum_{n=0}^{\infty} C_{n,k} (z-\zeta)^n$ all converge in a disk centered at~$\zeta$ and extending at least up to the closest other singular point.
Thus, the expression~\eqref{eq:regsing} defines an analytic function on a \emph{slit disk}~$\Delta_{\zeta}$ around $\zeta$ (a disk with a line segment from the center of the disk to the boundary removed).

\begin{remark} 
We always take $\log$ to mean the principal branch of the complex logarithm, defined by
\begin{equation} \label{eq:log}
  \log(r e^{i \theta}) = \log r + i \theta
  \quad \text{for $r > 0$ and $-\pi < \theta \leq \pi$.}
\end{equation}
The cut in~$\Delta_{\zeta}$ then points to the left, and any solution defined in a sector with apex at~$\zeta$ that does not intersect $\zeta + \R_{<0}$ has a singular expansion as a finite sum of terms of the form~\eqref{eq:regsing}, possibly with different~$\nu$.
\end{remark}

Methods dating back to Frobenius allow one to compute local series expansions of this type to any order for a basis of solutions (see \cite[Ch.~V]{Poole1936} for details).

\begin{example}
The point $z=1/2$ does not lie in $\Xi$ so, by Cauchy's theorem, all solutions of~\eqref{eq:mainDiff} have convergent power series expansions in disks around $z=1/2$. Continuing from the Sage commands above, running
\begin{Verbatim}
sage: deq.local_basis_expansions(1/2, order=5)
\end{Verbatim}
returns three truncated expansions
\begin{align*}
&\textstyle 1 - \frac{23463856}{144207}\left(z - \frac{1}{2}\right)^3 + \frac{1484883560}{1009449}\left(z - \frac{1}{2}\right)^4 + \cdots 
\\
&\textstyle \left(z - \frac{1}{2}\right) - \frac{10851808}{144207}\left(z - \frac{1}{2}\right)^3 + \frac{706529240}{1009449}\left(z - \frac{1}{2}\right)^4 + \cdots
\\
&\textstyle \left(z - \frac{1}{2}\right)^2 - \frac{97280}{6867}\left(z - \frac{1}{2}\right)^3 + \frac{5204338}{48069}\left(z - \frac{1}{2}\right)^4 + \cdots
\end{align*}
which begin convergent power series expansions at $z=1/2$ for a basis to the space of solutions defined on a small disk around~$1/2$.
\end{example}

\begin{example}
\label{ex:Abasis}
The point $z=0$ lies in $\Xi$, so solutions of~\eqref{eq:mainDiff} may have singularities at the origin. The command
\begin{Verbatim}
sage: deq.local_basis_expansions(0, order=3)
\end{Verbatim}
now returns truncated expansions
\begin{equation}
\begin{split}
A_1(z) &= \textstyle z^{-1}\log z - 9(\log z)^2 + 141\log z + z\left(\frac{475}{12} - \frac{483}{2} \log^2 z + 3471\log z\right) + \cdots\\
A_2(z) &= \textstyle z^{-1} - 18\log z + z\left(\frac{625}{2} - 483z\log z\right) + \cdots \\
A_3(z) &= \textstyle 1 + \frac{161}{6}z + \cdots
\end{split}
\label{eq:Abasis}
\end{equation}
for series converging in $\{ z: |z| < 3 - 2\sqrt{2}, z \notin \R_{\leq 0} \}$ which form a basis to the solution space of the differential equation.
Because the formal series $f(z)$ satisfies~\eqref{eq:mainDiff} it converges at the origin and can be written as a $\C$-linear combination of the $A_j$. Since $f(z)$ involves no logarithmic terms, and $f(0)=72$, we can represent $f$ in the $A_j$ basis as 
\[ f(z) = 0\cdot A_1(z) + 0\cdot A_2(z) + 72\cdot A_3(z). \qedhere \]
\end{example}

As stated above, we wish to find the singularity of $f(z)$ of minimal modulus, so we let $\rho = 3 - 2\sqrt{2}$ be the non-zero element of $\Xi$ with minimal modulus. 

\begin{example}
\label{ex:Bbasis}
The commands
\begin{Verbatim}
sage: rho = QQbar(3-2*sqrt(2))
sage: deq.local_basis_expansions(rho, order=3)
\end{Verbatim}
return truncated expansions
\begin{equation}
\begin{split}
B_1(z) &= \textstyle (z - \rho)^{-4}\log(z - \rho) - (z - \rho)^{-3}\left(\frac{5\sqrt{2}}{8} + 1 + \frac12 \log(z - \rho)\right) + \cdots\\
B_2(z) &= \textstyle (z -\rho)^{-4} - \frac{1}{2}(z - \rho)^{-3} + \cdots \\
B_3(z) &= \textstyle 1 - \left(\frac{5}{\sqrt{2}} + \frac{9}{2}\right)\left(z - \rho\right) + \cdots
\end{split}
\label{eq:Bbasis}
\end{equation}
for a basis of formal solutions at~$z=\rho$ of~\eqref{eq:mainDiff}.
These formal series converge in a disk around $z=\rho$ slit along the half-line $(-\infty, \rho]$.
Running the same command without the \Verb|order| parameter reveals that the terms not displayed here also involve $\log(z - \rho)^2$
and shows that no higher powers of $\log(z - \rho)$ can appear; i.e., in the notation of~\eqref{eq:regsing} we have $\kappa=2$.
\end{example}

\begin{remark} \label{rk:log}
The ore\_algebra package returns singular expansions which are linear combinations of powers of $(z-\zeta)$ and $\log(z-\zeta)$.
For singularity analysis, however, it is convenient to represent these expansions as linear combinations of powers of $(z-\zeta)$ and $\log\bigl(1/(1-z/\zeta)\bigr)$, so as to obtain expressions that are analytic in a slit neighbourhood of~$\zeta$ with the cut pointing away from~$0$.
For general $z$~and~$\zeta$, according to~\eqref{eq:log}, one has
\[
  \log \frac1{1-z/\zeta}
  = \log(-\zeta) - \log(z-\zeta) +
    \begin{cases}
      +2 \pi i, &0 < \arg(\zeta) \leq \arg(z - \zeta), \\
      -2 \pi i, &0 \geq \arg(\zeta) > \arg(z - \zeta), \\
      0, &\text{otherwise}.
    \end{cases}
\]
In the special case $\zeta = \rho$,
we obtain $\log((1-z/\rho)^{-1}) = \log(-\rho) - \log(z - \rho) + L$
with $L = 0$ when $\Im(z) \geq 0$ and $L = -2\pi i$ when $\Im(z) < 0$.
\end{remark}

The transfer theorems of Flajolet and Odlyzko~\cite{FlajoletOdlyzko1990} show how dominant asymptotics of $d_n$ can be immediately deduced from the singular expansion of $f$ near $z=\rho$.
The transfer theorems apply because, by Lemma~\ref{lem:singularities}, the function~$f$ extends analytically to the domain
\begin{equation}
  \label{eq:Delta}
  \Delta = {\{ z : |z| < 1 \}} \setminus [\rho, 1].
\end{equation}

The functions $\tilde B_1, \tilde B_2, \tilde B_3$ obtained by replacing
$\log(z - \rho)$
by
$\log((1-z/\rho)^{-1})$
in~\eqref{eq:Bbasis} form a basis of the solution space of~\eqref{eq:mainDiff} in a neighbourhood of~$\rho$ in~$\Delta$, and to determine asymptotics it is sufficient to represent~$f$ in the $\tilde B_j$~basis. Example~\ref{ex:Abasis}, which expressed $f$ in the $A_j$ basis, crucially relied on our knowledge of $f(z)$ near the origin, supplied by its power series coefficients $d_n$.
This argument does not apply at any non-zero point.
Fortunately, it is possible to compute the change of basis matrix between the $A_j$ and the~$\tilde B_j$ when viewed as solutions of~\eqref{eq:mainDiff} on the same domain contained in~$\Delta$.
By Remark~\ref{rk:log} each $\tilde B_j$ coincides with~$B_j$ in the upper half-plane, so for practical reasons we compute the change of basis matrix between the $A_j$ and $B_j$ bases.
This is implemented in ore\_algebra using rigorous numeric analytic continuation along a path.

The ore\_algebra package uses numeric approximations of real numbers certified to lie in intervals, as implemented in the Arb library~\cite{Johansson2017}.
In what follows, any expression of the form $[x \pm \epsilon]$ for $x\in\R$ and $\epsilon\geq0$ refers to an exact constant which is known to lie in the interval $[x + \epsilon, x - \epsilon]$.
The values displayed in the text are low-precision over-approximations of the intervals used in the actual computation.

\begin{example}
We select an analytic continuation path that goes from~$0$ to~$\rho$ without leaving the domain~$\Delta$,
and, because of the relation between $B_j$~and~$\tilde B_j$, that arrives at~$\rho$ from the upper half-plane. Using the polygonal path $\gamma = (0, i, \rho)$ for the required analytic continuation, the command 
\begin{Verbatim}
sage: M = deq.numerical_transition_matrix(path=[0, I, rho], eps=1e-20)
sage: [lambda1, lambda2, lambda3] =  M * vector([0, 0, 72])
\end{Verbatim}
computes the change of basis $M$ from the $A_j$ to the $B_j$ basis, then determines the rigorous approximations 
\begin{align*}
\lambda_1 &= [-0.0420 \pm 3.14\cdot 10^{-5}] + [\pm 1.21\cdot 10^{-14}]\,i, \\
\lambda_2 &= [-0.0141 \pm 3.22\cdot 10^{-5}] + [0.132 \pm 1.52\cdot 10^{-4}]\,i, \\
\lambda_3 &= [-12.5 \pm 0.0407] + [26.8 \pm 0.0117]\,i
\end{align*}
to the constants $\lambda_1,\lambda_2,\lambda_3$ such that
\[f(z) = \lambda_1B_1(z) + \lambda_2B_2(z) + \lambda_3B_3(z)
      = \lambda_1 \tilde B_1(z) + \lambda_2 \tilde B_2(z) + \lambda_3 \tilde B_3(z),\]
where all functions are implicitly extended by analytic continuation along~$\gamma$.
The expansions~\eqref{eq:Bbasis} from Example~\ref{ex:Bbasis} then give the initial terms of a singular expansion
\begin{equation*}
\begin{split}\
f(z) &= \bigl( [0.0598 \pm 4.79\cdot 10^{-5}] + [\pm 9.21\cdot 10^{-14}]\,i \bigr) (z-\rho)^{-4} \\
& + \bigl( [0.0420 \pm 3.14\cdot 10^{-5}] + [\pm 1.21\cdot 10^{-14}]\,i \bigr) (z-\rho)^{-4}\log\frac{1}{1-z/\rho}
+ \cdots,
\end{split}
\end{equation*}
where `$\cdots$' hides terms with factors $(z-\rho)^{\alpha} \log(z - \rho)^{\beta}$ where $\alpha \geq -3$ and $\beta \leq 2$.
Since $f$ is a real function the imaginary parts appearing in the coefficients are exactly zero, and
\begin{equation}
  \label{eq:fsingexpansion}
  f(z) = C_1 (z-\rho)^{-4} + C_2 (z-\rho)^{-4}\log\frac{1}{1-z/\rho} + \cdots
\end{equation}
for constants
\begin{align*}
  C_1 &= [0.0598 \pm 4.79\cdot 10^{-5}],&
  C_2 &= [0.0420 \pm 3.14\cdot 10^{-5}].
\end{align*}
The fact that the computed intervals containing $C_1$ and $C_2$ do not contain zero confirms that the analytic function~$f$ is singular at~$\rho$.
\end{example}

Corollary~5 of Flajolet and Odlyzko~\cite{FlajoletOdlyzko1990} gives an explicit formula for dominant asymptotics of $d_n$ in terms of the constants in the singular expansion~\eqref{eq:fsingexpansion}, leading to dominant asymptotic behaviour
\begin{equation}
  \label{eq:asy}
    \begin{aligned} 
    d_n
      &= \rho^{n-4} \frac{n^3}{6} \bigl( C_1 + C_2 (\log n - \gamma - 11/6) \bigr)
	  + O\left(\rho^{-n}n^2\log^2 n\right) \\
      &=  [8.07 \pm 1.96\cdot 10^{-3}] \; \rho^{-n}n^3\log n + [1.37 \pm 1.41\cdot 10^{-3}] \; \rho^{-n}n^3
      + O\left(\rho^{-n}n^2\log^2 n\right),
    \end{aligned}
\end{equation}
where $\gamma = [0.58 \pm 3.83\cdot 10^{-3}]$ is the Euler-Mascheroni constant. Although we have not computed the constants in closed form, this expansion shows that $d_n$ is \emph{eventually} positive. 
\begin{proposition}[Eventual Positivity]
There exists $N\in\N$ such that $d_n >0$ for all $n>N$. 
\end{proposition}
Because Conjecture~\ref{conj:YuChen1} asks us to prove \emph{all} terms of $d_n$ are positive, we must delve deeper. We determine a precise natural number $N$ such that the positive leading asymptotic term dominates the error in the asymptotic approximation for $n>N$, then computationally check the finite number of remaining values.

\section{Complete Positivity}

Our proof mirrors the constructive proofs of transfer theorems for asymptotic behaviour of sequences by Flajolet and Odlyzko~\cite{FlajoletOdlyzko1990}.
The starting point is the Cauchy integral formula.
Since $f$~is analytic on the domain~$\Delta$ defined in Equation~\eqref{eq:Delta},
the Cauchy integral formula gives the representation
\[ d_n = \frac{1}{2\pi i}\int_{|z|=\delta} \frac{f(z)}{z^{n+1}}dz \]
for any $0< \delta < \rho$ and all $n \geq 0$.
Asymptotic behaviour is determined by manipulating the domain of integration $\{|z|=\delta\}$ without crossing the singularities of the integrand, in such a way that the integral over part of the domain of integration is negligible while integration over the remaining part can be approximated by replacing $f(z)$ by its singular expansion at its singularity $z=\rho$ closest to the origin. 

Towards our explicit asymptotic bounds, let $\ell(z)$ denote the leading term in the singular expansion~\eqref{eq:fsingexpansion} of $f(z)$ at $z=\rho$, meaning
\begin{align*}
  \ell(z) = C_1 (z-\rho)^{-4} + C_2 (z-\rho)^{-4}\log\frac{1}{1-z/\rho}
\end{align*}
for the constants $C_1$ and $C_2$ in the singular expansion~\eqref{eq:fsingexpansion}. This expansion implies the existence of functions $h_0(z),h_1(z),$ and $h_2(z)$, analytic at $z=\rho$, such that 
\begin{equation}
f(z) = \ell(z) + \underbrace{(z-\rho)^{-3}\left(h_0(z) + h_1(z)\log\frac{1}{1-z/\rho} + h_2(z)\log^2\frac{1}{1-z/\rho}\right)}_{g(z)}.
\label{eq:hfunctions}
\end{equation}
Series expansions of the $h_j$ at $z=\rho$ can be computed to arbitrary order with coefficients rigorously approximated to any precision using series expansions of the $B_j$ basis at $z=\rho$ and the change of basis matrix $M$ from above.

\begin{remark} \label{rk:analyticity domains}
Since the origin is the closest element of $\Xi$ to $\rho$, the functions $h_0,h_1,$ and $h_2$ appearing in~\eqref{eq:hfunctions} are analytic on the disk
$|z - \rho| < \rho$.
Because $f$~and~$\ell$ are both analytic on~$\Delta$, so is $g = f - \ell$.
\end{remark}

Write
\[ d_n = \frac{1}{2\pi i}\int_{|z|=\delta} \frac{\ell(z)}{z^{n+1}}dz + \frac{1}{2\pi i}\int_{|z|=\delta} \frac{g(z)}{z^{n+1}}dz. \]
Behaviour of the first integral, which equals the $n$th power series coefficient of $\ell(z)$, is easily lower-bounded using standard generating function manipulations.

\begin{proposition} 
\label{prop:lowerbound}
For all $n\in\N$,
\[ \frac{1}{2\pi i}\int_{|z|=\delta} \frac{\ell(z)}{z^{n+1}}dz
\geq \rho^{-n} n^3 \, (8.07\, \log n + 1.37).\]
\end{proposition}
\begin{proof}
Proposition~\ref{prop:lowerbound} is proven in Section~\ref{sec:lowerboundproof}.
\end{proof}

After lower-bounding the integral of the leading term $\ell(z)$, which is positive for all $n$, we turn to upper-bounding the integral of the remainder $g(z)$.

\begin{proposition} 
\label{prop:upperbound}
For all integers $n\geq 1000$,
\[
  \left| \frac{1}{2\pi i}\int_{|z|=\delta} \frac{g(z)}{z^{n+1}}dz \right|
  \leq 1196 \, \rho^{-n} n^2 \log^2 n.
\]
\end{proposition}
\begin{proof}
Proposition~\ref{prop:upperbound} follows from Propositions~\ref{prop:upperboundS},~\ref{prop:upperboundL} (in the limit $\varphi \to 0$), and~\ref{prop:upperboundB} in Section~\ref{sec:upperboundproofs}. 
\end{proof}

This immediately gives an explicit bound where asymptotic behaviour implies sequence positivity.

\begin{corollary}
One has $d_n > 0$ for all $n \in \N$.
\end{corollary}
\begin{proof}
Propositions~\ref{prop:lowerbound} and~\ref{prop:upperbound} imply that
\[ d_n \geq \rho^{-n} n^2 \log^2 n \, \left(8.07\,\frac{n}{\log n} + 1.37\,\frac{n}{\log^2 n} - 1196\right) \]
for all $n \geq 1000$.
The final factor is increasing for $n \geq 8$ and positive at~$n=1000$, hence $d_n > 0$ for $n \geq 1000$.
One can explicitly check that $d_n > 0$ for $0 \leq n < 1000$.
\end{proof}

\subsection{Lower-Bounding the Leading Term Integral}
\label{sec:lowerboundproof}

If $a(z)$ is a complex-valued function analytic at the origin, we write $[z^n]a(z)$ for the $n$th term in the power series expansion of $a(z)$ centered at $z=0$.

\begin{proof}[Proof of Proposition~\ref{prop:lowerbound}]
Differentiating the geometric series $(1-z)^{-1} = \sum_{n \geq 0}z^n$ three times with respect to $z$ implies
\[ [z^n]\left(1-z\right)^{-4} = \frac{(n+1)(n+2)(n+3)}{6} \geq \frac{n^3}{6}, \]
while the identity
\[ \left(1-z\right)^{-4}\log\frac{1}{1-z} = \frac{d^3}{dz^3}\left(\frac{\log\bigl(1/(1-z)\bigr)}{6(1-z)} - \frac{11}{36(1-z)}\right) \]
implies 
\[ [z^n]\left(1-z\right)^{-4}\log\frac{1}{1-z} = \frac{(n+1)(n+2)(n+3)}{6}\left(H_{n+3} - \frac{11}{6}\right) ,  \]
where $H_n = \sum_{k=1}^n1/k$ is the $n$th harmonic number. Since $H_{n} \geq \log n + \gamma$, we obtain
\begin{align*}
[z^n]\ell(z)
&= [z^n]C_1(z-\rho)^{-4} + [z^n]C_2 (z-\rho)^{-4}\log\frac{1}{1-z/\rho} \\
&= C_1\rho^{-n-4}[z^n](1-z)^{-4} + C_2\rho^{-n-4}[z^n](1-z)^{-4}\log\frac{1}{1-z} \\
&\geq \left(\frac{C_1 + (\gamma-11/6) C_2}{6\rho^4}\right)n^3\rho^{-n} + \frac{C_2}{6\rho^4}n^3\rho^{-n}\log n.
\end{align*}
Note that this lower bound matches the leading asymptotic behaviour given in~\eqref{eq:asy}.
The proposition follows.
\end{proof}

\subsection{Upper-Bounding the Remainder Integral}
\label{sec:upperboundproofs}
\begin{figure}
  \centering
  \includegraphics{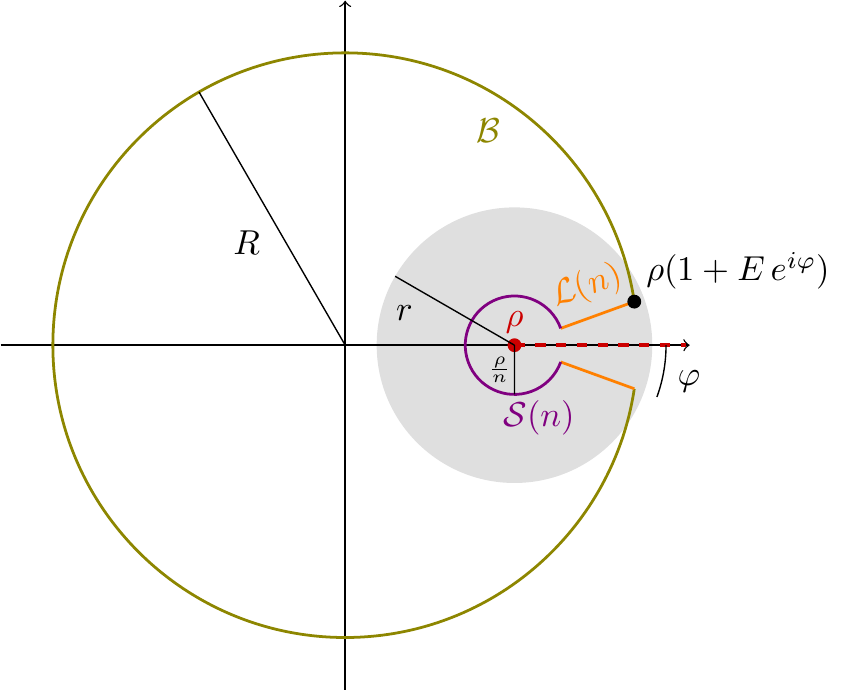}
  \caption{The Cauchy integral for the $n$th sequence term is deformed into the union of a big arc $\mB$ (in green), a small arc $\mS(n)$ (in purple), and two line segments $\mL(n)$ in orange. The series expansions of the basis~\eqref{eq:Bbasis} is defined on the disk $|z-\rho| < 1/8$ (in gray) with a line removed (in dashed red).}
  \label{fig:Domain}
\end{figure}

Following Flajolet and Odlyzko~\cite{FlajoletOdlyzko1990}, to upper-bound the integral of $g(z)$ we deform the domain of integration $|z|=\delta$, without crossing any singularities of the integrand, into 
\begin{itemize}
	\item An arc $\mB$ of a `big' circle of radius $R > \rho$,
	\item An arc $\mS(n)$ of a `small' circle of radius $\rho/n$,
	\item Two line segments $\mL(n)$ connecting the arcs of the big and small circles, supported by lines passing through~$\rho$ at small angles~$\pm\varphi$ with the positive real axis.
\end{itemize}
See Figure~\ref{fig:Domain} for an illustration.

To exploit the series expansions of the $h_j$ at $z=\rho$ we select $R$ so that, for large enough~$n$ and small enough~$\varphi$, the paths $\mS(n)$ and $\mL(n)$ lie within the disk of convergence of these expansions.
By Remark~\ref{rk:analyticity domains}, any $R < 2 \rho$ satisfies this constraint.
In view of the computation of error bounds, though, it is convenient to pick a radius~$r$ such that the
the punctured disk $0 < |z-\rho| < r$ does not contain any root of the leading polynomial of the differential equation~\eqref{eq:mainDiff}, and then choose~$R$ with $\rho < R < \rho + r$.
With this in mind, we take $r = 1/8 \approx 0.73 \rho$ and $R$ just smaller than~$\rho + r$.

\subsubsection{Bounding the Integrals Near the Singularity}

The first step towards our desired bounds is to upper-bound $h_0,h_1,$ and $h_2$ on the disk $|z-\rho|<r$. 

\begin{lemma}
\label{lem:upperboundondisk}
If $h_0,h_1,h_2$ are the functions defined in~\eqref{eq:hfunctions} then there exist constants
\begin{align}
b_0 &= [6.86 \pm 2.71\cdot 10^{-4}], &
b_1 &= [2.85 \pm 3.20\cdot 10^{-3}] &
b_2 &= [0.309 \pm 2.78\cdot 10^{-4}]
\label{eq:bconstants}
\end{align}
such that $|h_j(z)| \leq b_j$ for all $0 \leq j \leq 2$ and $z \in \{|z-\rho|<r\}$.
\end{lemma}

\begin{proof}
The bounds are computed using the implementation in ore\_algebra of the algorithm described in~\cite{Mezzarobba2019}.
Full details can be found in the accompanying Sage notebook.

In summary, the computation goes as follows.
Write the singular expansion of~$f$ at $\rho$ in the form
\[
  f(\rho + w) = \ell(\rho + w) + w^{-4} \left(
    u_0(w) + u_1(w) \log w + u_2(w) \frac{\log^2 w}{2}
  \right)
\]
for $\Im(w) > 0$.
In addition to the change of variable $z = \rho + w$, note that the logarithmic factors
$\log^j(w)/j!$
differ from the
$\log^j(1/(1-z/\rho)) = \log(-\rho) - \log w$
appearing in the definition~\eqref{eq:fsingexpansion} of the~$h_j$,
and that the polar factor $(z - \rho)^{-3}$ has become~$w^{-4}$, so that $u_j(0) = 0$ for all~$j$.

Let $\tilde u_j(w) = c_{j,1} \, w  + \dots + c_{j,49} \, w^{49}$ be the truncation at order~$50$ of the series expansion of~$u_j$.
We first compute numeric regions containing the coefficients of~$\tilde u_j$ and let
$m_0 = \sum_{k=1}^{49} \max_j |c_{j,k}| \, r^{k-1}$,
so that $|w^{-1} \, \tilde u_j(w)| \leq m_0$ in the disk $|w| \leq r$.

The next step is to bound the `tails' $u_j(w) - \tilde u_j(w)$.
After changing $z$ to $\rho + w$ in the differential equation~\eqref{eq:mainDiff}, we apply  \cite[Algorithm~6.11]{Mezzarobba2019} to the resulting differential operator.
The other relevant parameters are set to~$\lambda=-4$, $N=50$, and
$u_{-4+k} = c_{0, k} + c_{1, k} \log w + (c_{2, k}/2) \log^2 w$
for $k = 49, 48, \dots$,
using the previously computed~$c_{j,k}$.
The algorithm returns an expression~$\hat u(w)$ such that, 
by \cite[Proposition~6.12]{Mezzarobba2019}, the power series expansion
$\hat u(w) = \hat c_0 + \hat c_1 w + \cdots$
satisfies
$|c_{j,k}| \leq \hat c_k$
for $j = 0, 1, 2$ and $k \geq 50$.
In addition, it can be seen from the way $\hat u$ is constructed in the algorithm that $\hat c_0 = 0$ (in fact, $c_k = 0$ for $k \leq 49$).
We evaluate $w^{-1}\, \hat u(w)$ at $w=r$ using
\cite[Algorithm~8.1]{Mezzarobba2019}
to obtain, by the triangle inequality, a bound $m_1$ such that
$\bigl|w^{-1}\bigl(u_j(w) - \tilde u_j(w)\bigr)\bigr| \leq m_1$
for $|w| \leq r$.

Adding the bounds, we have
$|w^{-1} \, u_j(w)| \leq m_0 + m_1$
for $|w| \leq r$.
Finally, letting $a = \log(-\rho)$, the expressions of the~$h_j$ in terms of the~$u_j$ read
\[
  h_0(z) = w^{-1} \left(u_0(w) + a u_1(w) + \frac{a^2 u_2(w)}{2}\right), \quad
  h_1(z) = w^{-1} \bigl(-u_1(w) - a u_2(w) \bigr), \quad
  h_2(z) = w^{-1} \frac{u_2(w)}{2}.
\]
We can hence take
$b_j = d_j\,(m_0 + m_1)$
where
$d_0 = (|a| + |a|^2/2)$,
$d_1 = (1 + |a|)$,
and
$d_2 = 1/2$.
\end{proof}

\begin{definition}
Let $B$ be the quadratic polynomial $B(z) = b_0 + b_1z + b_2z^2$, where the $b_0,b_1,$ and $b_2$ are the constants in~\eqref{eq:bconstants}.
\end{definition}

The bounds on the $h_j(z)$ in Lemma~\ref{lem:upperboundondisk} allow us to bound the integrals of $g(z)$ over $\mS(n)$ and $\mL(n)$. 

\begin{proposition}
\label{prop:upperboundS}
For all integers $n \geq 5$,
\[ \left|\frac{1}{2\pi i}\int_{\mS(n)} \frac{g(z)}{z^{n+1}}dz\right|
\leq \rho^{-n} n^2\; \frac{4}{\rho^3}B(\pi + \log n). \]
\end{proposition}
\begin{proof}
Let $n \geq 5$.
Parametrizing $|z-\rho|=\rho/n$ by $z = \rho + \rho e^{i\theta}/n$ we have $|z| \geq \rho(1-1/n)$ and 
\[ \left|\log\frac{1}{1-z/\rho}\right| = \left|\log\left(e^{-i\theta}\right) - \log n\right| \leq \pi + \log n,\]
so, using the fact that $\rho/n < r$,
\begin{align*}
\left|\frac{1}{2\pi i}\int_{\mS(n)} \frac{g(z)}{z^{n+1}}dz\right|
&\leq \frac{\text{length}(\mS(n)) \; (n/\rho)^3}{2\pi\rho^{n+1}(1-1/n)^{n+1}} \max_{z \in \mS(n)} \left|h_0(z) + h_1(z)\log\frac{1}{1-z/\rho} + h_2(z)\log^2\frac{1}{1-z/\rho}\right| \\[+2mm]
&\leq \rho^{-n}n^2(1-1/n)^{-n-1} \rho^{-3}\left(b_0 + b_1(\pi+\log n) + b_2(\pi+\log n)^2\right).
\end{align*}
The factor $(1-1/n)^{-n-1}$ is decreasing, and less than~$4$ for $n=5$.
\end{proof}

\begin{proposition}
\label{prop:upperboundL}
For all integers $n\geq 2$ and all small enough~$\varphi$,
\[ \left|\frac{1}{2\pi i}\int_{\mL(n)} \frac{g(z)}{z^{n+1}}dz\right| \leq  \rho^{-n}n^2 \cdot \frac{B(\pi + \log n)}{\pi\rho^3 \cos\varphi}. \]
\end{proposition}
\begin{proof}
Fix $n \geq 2$.
The integral over the upper part of~$\mL(n)$ equals
\[
L_+(n) = \frac{1}{2\pi i} \int_{\rho(1+e^{i\varphi/n})}^{\rho(1+E\,e^{i\varphi/n})} \frac{g(z)}{z^{n+1}} dz
= 
\frac{1}{2\pi i} \sum_{j=0}^2 \int_{\rho(1+e^{i\varphi/n})}^{\rho(1+E\,e^{i\varphi/n})} \frac{h_j(z) \log^j\displaystyle\frac{1}{1-z/\rho}}{(z-\rho)^3z^{n+1}} dz
\]
for some $E \geq r$ (depending on~$\varphi$ but not on~$n$).
The substitution $z=\rho (1 + e^{i \varphi} t/n)$ yields
\[
  L_+(n)
  = \frac{1}{2\pi}
    \sum_{j=0}^2 \int_{1}^{E\,n}
      \frac
	{h_j\bigl(\rho (1 + e^{i\varphi} t/n)\bigr) \log^j(-e^{i\varphi} n/t)}
        {(\rho e^{i\varphi} t/n)^3 \rho^{n+1} (1 + e^{i\varphi} t/n)^{n+1}}
      \frac{\rho e^{i\varphi}}n dt.
\]
When $\varphi > 0$ is small enough, one has
$\log(-e^{i\varphi} n/t) = i (\varphi - \pi) + \log(n/t)$,
and the integration segment is contained in the disk $|z - \rho| \leq r$,
so that $|h_j(z)| \leq b_j$ in the integrand.
Therefore the modulus of the integral satisfies
\[
  |L_+(n)|
  \leq
    \rho^{-n}n^2
      \cdot \frac{B(\pi + \log n)}{2\pi\rho^3}
      \cdot \int_1^\infty t^{-3} \left(1 + \frac{t \cos \varphi}{n}\right)^{-n-1} dt
\]
where
\[
  \int_1^\infty t^{-3} \left(1 + \frac{t \cos \varphi}{n}\right)^{-n-1} dt
  \leq \int_1^\infty \left(1 + \frac{t \cos \varphi}{n}\right)^{-n-1} dt
  = \frac{1}{\cos \varphi} \left(1 + \frac{\cos \varphi}{n}\right)^{-n}.
\]
The right-hand side is decreasing, and is bounded by $1/\cos\varphi$ as soon as $n \geq 2$ and $\varphi < \pi/3$.
The same reasoning applies to the integral over the other part of $\mL(n)$, with the sole difference that $\varphi$ is replaced by~$-\varphi$, so that the logarithmic factor in the integrand becomes
$i (- \varphi + \pi) + \log(n/t)$.
\end{proof}

\subsubsection{Bounding the Integral on the Big Circle}
\begin{figure}
\centering
\includegraphics{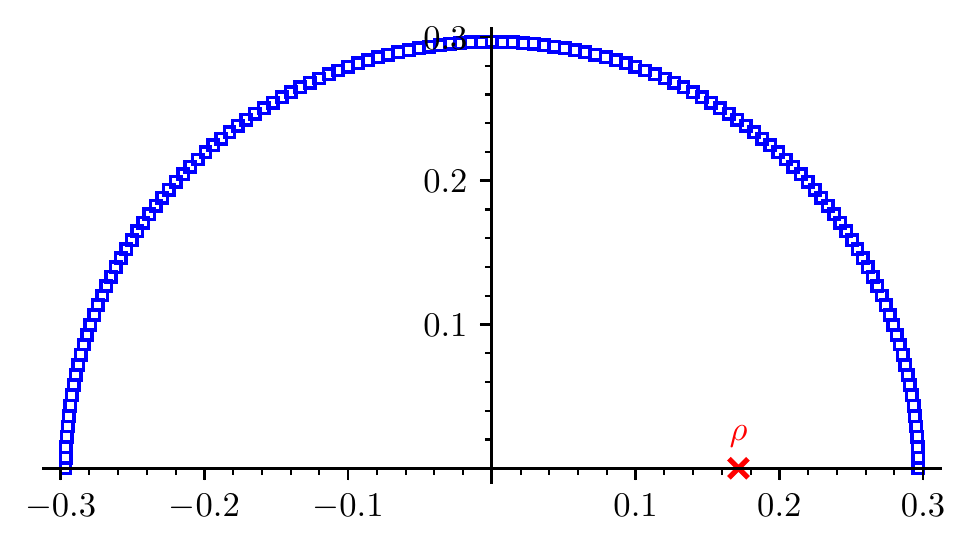}
\caption{The overlapping rectangles used to establish Proposition~\ref{prop:upperboundB}.}
\label{fig:rectangles}
\end{figure}

Finally, we can bound the integral over the big circle.
\begin{proposition}
\label{prop:upperboundB}
For all $n\in\N$, 
\[ \left|\frac{1}{2\pi i}\int_{\mB} \frac{g(z)}{z^{n+1}}dz\right| \leq 1753.15 \, R^{-n}. \]
\end{proposition}
\begin{proof}
Standard integral bounds imply 
\[ \left|\frac{1}{2\pi i}\int_{\mB} \frac{g(z)}{z^{n+1}}dz\right| \leq R^{-n} \cdot \max_{z \in \mB} |g(z)| = R^{-n} \cdot \max_{z \in \mB} |f(z)-\ell(z)|. \]
Since we know $\ell(z)$ in closed form, the stated upper bound follows bounding $f(z)$ on the circle $|z|=R$. In fact, because $\overline{f(z)}=f(\overline{z})$ it is sufficient to upper bound $f(z)-\ell(z)$ on the upper half of $|z|=R$. This is accomplished by covering this half-circle by overlapping rectangles with rational coordinates, displayed in Figure~\ref{fig:rectangles}, then rigorously computing bounds for $f(z)$ and $\ell(z)$ on these rectangles.

Numeric regions containing $f(z)$ on each rectangle are computed in Sage using the \Verb#numerical_solution()# method of differential operators to solve the differential equation~\eqref{eq:mainDiff} in interval arithmetic.
This method implements a strategy very similar to the one we employed to bound the functions~$h_j$ in the proof of Lemma~\ref{lem:upperboundondisk}---but limited to the simpler case where the function to be evaluated is a solution of the differential equation over a domain free of singularities, as opposed to a function obtained starting from a solution by factoring out a singular part.
\end{proof}

\section{Further Remarks}
\label{sec:finalremarks}
We end with some final remarks.

\begin{remark}
Yu and Chen~\cite{YuChen2020} give the sequence $(d_n)$ as a nested sequence of binomial sums. Such a sequence can be algorithmically written as the diagonal of a multivariate rational power series~\cite{BostanLairezSalvy2017}, and then for sufficiently large $n$ as an explicit multivariate saddle-point integral~\cite{PemantleWilson2013,Melczer2020}. It is theoretically possible to prove Theorem~\ref{thm:mainthm} through explicit bounds for such saddle-point integrals; this approach is less practical than going through the singularity analysis above but would give explicit constants (instead of certified intervals) for the leading asymptotic terms of $d_n$. A hybrid approach, using multivariate techniques to derive the leading asymptotic term with explicit coefficients then using the differential equation to bound some of the sub-dominant terms, is also possible.
\end{remark}

\begin{remark}
\label{rem:fpositive}
While the proof of Conjecture~\ref{conj:YuChen2} is interesting in its own right,
Yu and Chen's uniqueness result only requires that the function~$f(z)$ takes positive values on the real interval $z \in (0, \rho)$ \cite[Sec.~1.3, III]{YuChen2020}.
This weaker statement is easier to prove using rigorous numerics than the positivity of the coefficient sequence.
The idea is to split the interval $[\varepsilon, \rho - \varepsilon]$ into subintervals over which we can evaluate~$f$ accurately enough to check that it is positive,
handling the limits $z \to 0$ and $z \to \rho$ as in the proof of Lemma~\ref{lem:upperboundondisk}.

The presence of an apparent singularity $z_0 = 0.019\dots$ of~\eqref{eq:mainDiff} in the interior of the interval causes a small complication, for \Verb#numerical_solution()# currently does not support evaluation on non-point intervals containing singular points.
One way around the issue would be to treat this singularity like~$0$ and~$\rho$.
As a quicker alternative, we perform a \emph{partial desingularization} of~\eqref{eq:mainDiff}, yielding a new equation satisfied by~$f$ that does not have~$z_0$ as a singularity while not being as large and difficult to solve numerically as the fully desingularized equation of Lemma~\ref{lem:singularities}.

Using this new equation, no additional subinterval besides the neighborhoods of $0$ and~$\rho$ turns out to be necessary.
Indeed, one can show that the tail $\sum_{n=58}^{\infty} d_n z^n$ of the series expansion of~$f$ at the origin is bounded by~$1.71$ for $|z| \leq r_0 = 0.0675$.
As $d_0 = 72$ and we have already checked that $d_n > 0$ for all $n \leq 1000$, this implies that $f(z) > 0$ for $0 < z < r_0$.
Then, reusing the results of the computations done for the proof of Lemma~\ref{lem:upperboundondisk} and its notation, one has
$|u(w) - \tilde u(w)| \leq \hat u(\rho-r_0) \leq m = 7.82 \cdot 10^{-9}$
for $w \leq \rho - r_0$.
We rewrite the local expansion of~$f$ in terms of $-w = \rho - z$ and $- \log(-w)$, both positive for $r_0 < z < \rho$, and subtract
$m\, \bigl(1 - \log(\rho - z) - (1/2) \log^2(\rho - z)\bigr)$
from its explicitly computed order-50 truncation to obtain a lower bound on $f(z)$.
This lower bound is an explicit polynomial in~$w$ and $\log(-w)$ that can be verified to take positive values for $r_0 - \rho < w < 0$.
Details of the calculations can be found in the accompanying Sage notebook.
\end{remark}

\begin{remark}
The method employed here to study the sequence~$(d_n)$ can be used, more generally, to produce
approximations with error bounds
\[
  u_n = \rho^{-n} \, n^{\alpha} \sum_{k=0}^K \sum_{j=0}^J [c_{k,j} \pm \varepsilon_{k,j}] \, \frac{\log^j n}{n^k},
  \qquad
  n \geq n_0,
\]
of sequences~$u_n$ whose generating series satisfy linear differential equations with polynomial coefficients and regular dominant singularities,
though it is bound to yield trivial results in some ``difficult'' cases due to the decidability issues mentioned in Section~\ref{sec:history}.
It would be interesting to understand exactly how general it can be made,
and how to turn it into a practical algorithm that would automatically choose judicious values for all parameters.
\end{remark}

\section{Acknowledgments}
The authors thank Thomas Yu for bringing the uniqueness of the Canham model, and its connection to integer sequence positivity, to our attention.

\bibliographystyle{alpha}
\bibliography{bibl}

\section*{Appendix}
\label{sec:appendix}
Here we list an explicit recurrence and differential equation satisfied by the sequence $(d_n)$ and its generating function $f(z)$, respectively. The sequence $(d_n)$ satisfies the recurrence
\begin{equation}
0= \sum_{k=0}^7 r_k(n)d_{n+k}
\label{eq:mainRec}
\end{equation}
where
{\scriptsize
\begin{align*}
r_0(n) &= -(n+8)(n+7)(12232n^3+298144n^2+2412586n+6469077)(n+6)^2\\
r_1(n) &=  (n+8)(183480n^6+7655560n^5+131977142n^4+1202876299n^3+6112196895n^2+16418149668n+18219511026)\\
r_2(n) &=  -(n+8)(941864n^6+38326904n^5+644300514n^4+5727711699n^3+28407144241n^2+74557779538n+80949464718)\\
r_3(n) &=  (1993816n^7+97303624n^6+2021855198n^5+23184921987n^4\\
&\hspace{2.35in}+158457515673n^3+645518710454n^2+1451619424860n+1390493835900)\\
r_4(n) &=  (-1993816n^7-98090344n^6-2054897438n^5-23758375953n^4\\
&\hspace{2.4in}-163720428321n^3-672459054524n^2-1524577250976n-1472211879228)\\
r_5(n) &= (n+6)(941864n^6+40789672n^5+730497394n^4+6921881565n^3+36590122947n^2+102300885158n+118218544398) \\
r_6(n) &= (n+6)(183480n^6+7756760n^5+135519142n^4+1252328453n^3+6456460129n^2+17612930492n+19872693550)\\
r_7(n) &= (n+7)(n+6)(12232n^3+215600n^2+1256970n+2435511)(n+8)^2
\end{align*}
}

\noindent
and the generating function $f(z)$ is a solution $F(z)=f(z)$ of the differential equation
\begin{equation}
0= \sum_{k=0}^3 c_k(z)F^{(k)}(z),
\label{eq:mainDiff}
\end{equation}
where
{\scriptsize
\begin{align*}
c_3(z) &= 8388593z^2(3z^4 - 164z^3 + 370z^2 - 164z + 3)(z + 1)^2(z^2 - 6z + 1)^2(z - 1)^3 \\
c_2(z) &= 8388593z(z + 1)(z^2 - 6z + 1)(66z^8 - 3943z^7 + 18981z^6 - 16759z^5 - 30383z^4 + 47123z^3 - 17577z^2 + 971z - 15)(z - 1)^2 \\
c_1(z) &= 16777186(z - 1)(210z^{12} - 13761z^{11} + 101088z^{10} - 178437z^9 - 248334z^8 \\
&\hspace{1.5in} + 930590z^7 - 446064z^6 - 694834z^5 + 794998z^4 - 267421z^3 + 24144z^2 - 649z + 6) \\
c_0(z) &= 6341776308z^{12} - 427012938072z^{11} + 2435594423178z^{10} - 2400915979716z^9 \\
& \hspace{1.5in} - 10724094731502z^8 + 26272536406048z^7 - 8496738740956z^6 - 30570113263064z^5 \\
& \hspace{1.5in} + 39394376229112z^4 - 19173572139496z^3 + 3825886272626z^2 - 170758199108z + 2701126946.
\end{align*}
}

\end{document}